\documentclass[11pt]{article}
\pdfoutput=1

\usepackage{graphicx}
\usepackage{amsmath}
\usepackage{amssymb,amsfonts}
\usepackage{amsthm}
\usepackage{cite}
\usepackage{hyperref}
\usepackage{bm}
\usepackage{mathrsfs}
\usepackage{amssymb}

\usepackage[margin=1in]{geometry}

\usepackage{algorithm}
\usepackage{algorithmic}

\usepackage{graphicx}
\graphicspath{{figures/}}

\usepackage[nameinlink]{cleveref}
\numberwithin{equation}{section}

\newtheorem{theorem}{Theorem}[section]

\newtheorem{proposition}[theorem]{Proposition}

\theoremstyle{definition}
\newtheorem{definition}[theorem]{Definition}

\theoremstyle{remark}

\begin{document}
	
	\title{A note on incremental POD algorithms for continuous time data}

\author{Hiba Fareed%
	\thanks{Department of Mathematics and Statistics, Missouri University of Science and Technology, Rolla, MO (\mbox{hf3n3@mst.edu}, \mbox{singlerj@mst.edu}).}
	\and
	John~R.~Singler%
	\footnotemark[1]
}

\maketitle

\begin{abstract}
	In our earlier work \cite[Fareed et al., Comput.\ Math.\ Appl.\ 75 (2018), no.\ 6, 1942-1960]{Fareed17}, we developed an incremental approach to compute the proper orthogonal decomposition (POD) of PDE simulation data.  Specifically, we developed an incremental algorithm for the SVD with respect to a weighted inner product for the discrete time POD computations.  For continuous time data, we used an approximate approach to arrive at a discrete time POD problem and then applied the incremental SVD algorithm.  In this note, we analyze the continuous time case with simulation data that is piecewise constant in time such that each data snapshot is expanded in a finite collection of basis elements of a Hilbert space.  We first show that the POD is determined by the SVD of two different data matrices with respect to weighted inner products.  Next, we develop incremental algorithms for approximating the two matrix SVDs with respect to the different weighted inner products.  Finally, we show neither approximate SVD is more accurate than the other; specifically, we show the incremental algorithms return equivalent results.
\end{abstract}

\textbf{Keywords:}  Proper orthogonal decomposition, continuous time, incremental SVD, weighted inner products

\textbf{Mathematics subject classifications (2010):}  65F30, 15A18

\section{Introduction}

Proper orthogonal decomposition (POD) is a data approximation technique that has been successfully used for many applications in various fields; see, e.g., \cite{Chaturantabut17, YangVeneziani17, XieMohebujjamanRebholzIliescu18, MohebujjamanRebholzXieIliescu17, GunzburgerJiangSchneier17, Kostova-VassilevskaOxberry18, Holmes12, Farhat15, Zimmermann14, Peng16, Kalashnikova14, Amsallem16, GubischVolkwein17, Stefanescu15, BaumannBennerHeiland18}.  The first part of any such application is to use POD to extract basis elements, called POD modes, from experimental or simulation data.  The POD modes are then used in various ways, such as forming optimal low order reconstructions of the data or constructing reduced order models of ordinary and partial differential equations (PDEs).

In the most basic case, the POD modes for a data set can be found using the singular value decomposition (SVD) of an appropriate data matrix.  Therefore, as the amount of data increases the computational cost and storage requirement for finding the POD modes also increase.  For this reason, researchers have investigated various approaches to lowering the computational cost and storage requirement for constructing the POD modes, the matrix SVD, and other related quantities \cite{Brand02, Brand06, BakerGallivan12, Chahlaoui03, Iwen16, Mastronardi05, Mastronardi08, Fahl01, BeattieBorggaard06, WangMcBeeIliescu16, HimpeLeibnerRave_pp}.  These more efficient algorithms have been used in various applications involving POD and other related approaches, such as the dynamic mode decomposition \cite{PlaczekTranOhayon11, Amsallem15, CoriglianoDossiMariani15, PeherstorferWillcox15, PeherstorferWillcox16, Schmidt17, ZahrFarhat15, Zimmermann17, ZimmermannPeherstorferWillcox17, NME:NME5283}.

We developed an incremental algorithm for POD computations in our earlier work \cite{Fareed17}.  In that algorithm, we considered simulation data arising from a Galerkin-type approximation method (e.g., a finite element method) for a PDE and updated the POD singular values and POD modes as new data became available.  Due to the class of PDE simulation data considered in \cite{Fareed17}, we developed an incremental SVD algorithm with respect to a weighted inner product to perform the POD computations.  The algorithm is computationally efficient, needs very little storage, and is also easily used in conjunction with an existing time stepping PDE approximation code.  We also recently performed an error analysis of the method in \cite{FareedSingler18pp}.

In both of these earlier works \cite{Fareed17, FareedSingler18pp}, we only performed analysis for the discrete time case, i.e., the data set is a finite collection of elements in a Hilbert space.  In \cite[Section 5]{Fareed17}, we developed an algorithm for the case of time varying data by approximating the POD integral operator using a Riemann sum and then performing an incremental POD/SVD update with respect to the weighted inner product.  In this note, we analyze the continuous time case assuming the data is piecewise constant in time.  First, in \Cref{sec:cont_disc_POD} we rigorously establish a precise relationship between the POD of the data and the SVD of two different matrices with respect to different weighted inner products.  In \Cref{sec:incr_SVD_weights}, we develop approximate incremental SVD algorithms for both cases and show that neither computed SVD is more accurate than the other; specifically, we show the incrementally computed approximate SVDs are equivalent.

\section{Background}
\label{Section:basic}

We begin by recalling material concerning the SVD of compact linear operators, the continuous time proper orthogonal decomposition, and the SVD of matrices with respect to weighted inner products.

\subsection{The SVD of a compact linear operator}
\label{subsec:SVD_operators}

In order to discuss the continuous time proper orthogonal decomposition, we first need to recall the singular value decomposition of a compact linear operator.

Let $ X $ and $ Y $ be separable Hilbert spaces with inner products $ (\cdot,\cdot)_X $ and $ (\cdot,\cdot)_Y $, and let $ A : X \rightarrow Y $ be a compact linear operator.  The Hilbert adjoint operator $ A^* : Y \rightarrow X $ is the compact linear operator satisfying $ (A x, y )_Y = ( x, A^* y )_X $ for all $ x \in X $ and $ y \in Y $.  The self-adjoint nonnegative compact operators $ AA^* : Y \to Y$ and $ A^*A : X \to X$ both have nonnegative eigenvalues and an orthonormal basis of eigenvectors.  The positive eigenvalues of these operators are equal and can be ordered as $ \lambda_1 \geq \lambda_2 \geq \cdots > 0 $.  The square roots of the positive eigenvalues are equal to the (ordered) positive singular values $ \{ \mu_k \} $ of $A$, and zero is included as a singular value of $ A $ if either $ A A^* $ or $ A^* A $ has a zero eigenvalue.  Denoting the orthonormal basis of eigenvectors of $ AA^*$ by $ \{\eta_k\} \subset Y $ and the orthonormal basis of eigenvectors of  $ A^*A$ by $ \{\xi_k\} \subset X $, we have the singular value expansions
\begin{equation*}
  A\xi = \sum_{k\geq1} \mu_k \,(\xi,\xi_k)_X\, \eta_k,  \quad    A^*\eta = \sum_{k\geq1} \mu_k \,(\eta,\eta_k)_Y\, \xi_k,
\end{equation*}
for all $ \xi \in X $ and $ \eta \in Y $.  This gives
\begin{equation}\label{eqn:svd_eqns}
  A\xi_i = \mu_i \eta_i,  \quad  A^*\eta_i = \mu_i \xi_i,  \quad  \forall  \mu_i > 0.
\end{equation}
Since only the positive singular values and corresponding singular vectors appear in the above formulas, we call these quantities the \textit{core} singular values and singular vectors.

For more information, see, e.g., \cite[Chapters VI--VIII]{GohbergGoldbergKaashoek90}, \cite[Chapter 30]{Lax02}, \cite[Sections VI.5--VI.6]{ReedSimon80}.

\subsection{Continuous time proper orthogonal decomposition}
\label{subsec:POD_continuous}

Let $ X $ be a separable Hilbert space with inner product $ (\cdot,\cdot)_X $ and corresponding norm $ \| \cdot \|_X $, and let $ I \subset \mathbb{R} $ be an interval.  Suppose $ u $ is a given element of the Bochner space $ L^2(I;X) $, i.e., roughly, $ \int_I \| u(t) \|_X^2 \, dt < \infty $.  The continuous time proper orthogonal decomposition problem is to find an orthonormal basis $ \{ x_k \} \subset X $ (called the POD modes) minimizing the data approximation error
$$
  E_r := \int_I \| u(t) - P_r u(t) \|^2_X \, dt,
$$
where $ P_r : X \to X $ is the orthogonal projection onto the first $ r $ basis elements, i.e.,
$$
  P_r u = \sum_{k=1}^r (u,x_k)_X x_k.
$$
The solution of the POD problem comes from the continuous time POD operator $ K : L^2(I) \to X $ for the data $ u $ defined by
\begin{equation}\label{equ:10}
  K f = \int_I u(t) f(t) \, dt.
\end{equation}
The POD operator is compact, and the Hilbert adjoint operator $ K^* : X \to L^2(I) $ is given by
\begin{equation}\label{eqn:POD_K_adjoint}
  (K^* x)(t) = \big( x, u(t) \big)_X.
\end{equation}
Let $ \{ \sigma_i, f_i, x_i \}_{i \geq 1} \subset \mathbb{R} \times L^2(I) \times X $ be the ordered singular values and corresponding orthonormal singular vectors.  The orthonormal basis minimizing the error is exactly given by the singular vectors $ \{ x_i \}_{i \geq 1} \subset X $, and the minimal value for the error is
$$
  E_r^\mathrm{min} = \sum_{i > r} \sigma_i^2.
$$

Given the data $ u $, a typical computation of the solution of the POD problem above focuses on finding the POD eigenvalues and modes $ \{ \lambda_i, x_i \}_{i \geq 1} $, where the POD eigenvalues are simply the squares of the POD singular values.  Also, in many applications the POD singular values for the data decay rapidly; therefore, only the first $ R $ singular values and modes are computed, where $ R $ is chosen so that $ E_R^\mathrm{min} $ or $ \sigma_R $ is small enough for the application.  If one is interested in approximately reconstructing the data (without storing the data), then one may desire to also compute the dominant $ L^2(I) $ POD singular vectors $ \{ f_i \}_{i \geq 1} $ since the optimal data approximation $ P_r u(t) $ can be rewritten as
$$
  P_r u(t) = \sum_{k=1}^r \sigma_k \, \overline{f_k(t)} \, x_k,
$$
where the bar denotes complex conjugate (for complex Hilbert spaces) \cite[Section 2.3]{Singler14}.

For more information, see, e.g., \cite{KunischVolkwein02,Singler11,Volkwein13,GubischVolkwein17}.

\subsection{The matrix SVD and weighted inner products}
\label{subsec:SVD_weight}

In continuous time POD applications, the given data $ u(t) $ is often finite dimensional; for example, the data is an approximate solution of a partial differential equation.  As we show in this note, for certain types of finite dimensional data the SVD required for the continuous time POD can be reduced to a matrix SVD with respect to weighted inner products.  Below, we review this type of matrix SVD, mostly following \cite{Fareed17}.

First, we give notation.  For a symmetric positive definite matrix $ M \in \mathbb{R}^{m \times m} $, let $ \mathbb{R}^m_M $ denote the Hilbert space $ \mathbb{R}^{m} $ with the $ M $-weighted inner product and corresponding norm given by $(x,y)_M = y^T M x$ and $ \| x \|_M = (x,x)_M^{1/2} = (x^T M x)^{1/2} $ for all $ x, y \in \mathbb{R}^m $.  The space $ \mathbb{R}^s_\Delta $ is defined in the same way, where $ \Delta $ is also symmetric positive definite.  Also, $\mathbb{R}^k $ without a subscript indicates the space is given the standard unweighted inner product and norm.

Next, for a matrix $ A \in \mathbb{R}^{m\times s} $ considered as a linear operator $ A : \mathbb{R}_{\Delta}^{s} \to \mathbb{R}_{M}^{m}$, the Hilbert adjoint operator $ A^{*} : \mathbb{R}_{M}^{m} \to \mathbb{R}_{\Delta}^{s}$ is the matrix $ A^{*} \in \mathbb{R}^{s \times m} $ satisfying
$$
  (Ax,y)_{M}=(x,A^{*}y)_{\Delta}  \quad  \mbox{for all $ x\in \mathbb{R}^{s}_\Delta $ and  $ y\in \mathbb{R}_{M}^{m} $.}
$$
It is straightforward to show $ A^{*} = \Delta^{-1} A^{T} M $.

Now we use the definition of the SVD of a compact operator in \Cref{subsec:SVD_operators} applied to the matrix $ A : \mathbb{R}_{\Delta}^{s} \to \mathbb{R}_{M}^{m} $.  Furthermore, we only focus on positive (or \textit{core}) singular values since zero singular values are typically not of interest in POD applications.  Suppose $ A $ has exactly $ k $ positive singular values $ \sigma_1\geq\sigma_2\geq \cdots \geq \sigma_k > 0 $.  Let $ V = [ v_1, v_2, \ldots, v_k ] \in \mathbb{R}^{m \times k} $ be the matrix whose columns are the first $ k $ $ M $-orthonormal eigenvectors of $ A A^{*} :  \mathbb{R}_{M}^{m} \to  \mathbb{R}_{M}^{m} $, and let $ W = [ w_1, w_2, \ldots, w_k ] \in \mathbb{R}^{s \times k} $ be the matrix whose columns are the first $ k $ $ \Delta $-orthonormal eigenvectors of $ A^{*} A : \mathbb{R}_{\Delta}^s \to \mathbb{R}_{\Delta}^s $.  \Cref{eqn:svd_eqns} gives
\begin{equation}\label{eqn:core_SVD_equivalent}
  A W = V \Sigma,  \quad  A^{*} V = W \Sigma,  \quad  \Sigma = \mathrm{diag}(\sigma_1, \ldots, \sigma_k).
\end{equation}
Since $ \{v_j\}_{j=1}^m $ and $ \{w_j\}_{j=1}^s $ are orthonormal in $ \mathbb{R}_{M}^{m} $ and $ \mathbb{R}_{\Delta}^{s} $, respectively, we have $ W^T \Delta W = I $ and $ V^T M V = I $.  Alternatively, we write $ W^* W = I $ and $ V^* V = I $, where $ W^* = W^T \Delta $ and $ V^* = V^T M $ are the Hilbert adjoint operators for $ W : \mathbb{R}^k \to \mathbb{R}^s_\Delta $ and $ V : \mathbb{R}^k \to \mathbb{R}^m_M $.  This implies \eqref{eqn:core_SVD_equivalent} is equivalent to
\begin{equation}\label{eqn:def_core_SVD}
  A = V \Sigma W^{*}.
\end{equation}

This leads to the following definition:
\begin{definition}
  For a matrix $ A: \mathbb{R}_{\Delta}^{s} \to \mathbb{R}_{M}^{m}$ with exactly $ k $ positive singular values, a \textit{core SVD} of $ A $ is given by $ A = V \Sigma W^{*} $, where $ V \in \mathbb{R}^{m \times k} $, $ \Sigma \in \mathbb{R}^{k \times k} $, and $ W \in \mathbb{R}^{s \times k} $ are defined above.
\end{definition}
See \cite{Fareed17} for more discussion about the core SVD.  If both inner products are unweighted, we call the core SVD the standard core SVD for clarity.

In the following result, we give a basic property of the core SVD.  The result is similar to Proposition 2.3 in \cite{Fareed17}.  The proof is also similar, and is omitted.
\begin{proposition} \label{prop:2.3}
	Suppose $ V_u \in \mathbb{R}^{m \times k} $ has $ M $-orthonormal columns and $ W_u \in \mathbb{R}^{s \times k} $ has $\Delta$-orthonormal columns.  If $ Q \in \mathbb{R}^{k \times k} $ has standard core SVD $ Q = V_Q \Sigma_Q W_Q^T $ and $ A: \mathbb{R}_{\Delta}^{s} \to \mathbb{R}_{M}^{m}$ is defined by $ A = V_u Q W_u^{*} $, then
	\begin{equation}\label{eqn:P_coreSVD_product}
	  A = V \Sigma_Q W^{*},  \quad  V = V_u V_Q,  \quad  W = W_u W_Q,
	\end{equation}
	is a core SVD of $ A $.
\end{proposition}

\textbf{Other notation:}  For a vector $ v \in \mathbb{R}^{n} $ and $ k \leq n $, let $ v_{(1:k)} $ denote the vector of the first $ k $ components of $ v $. Similarly, for a matrix $ A \in \mathbb{R}^{m \times n}$, let $ A_{(p : q, r : s)} $ denote the submatrix of $ A $ consisting of the entries of $ A $ from rows $ p, \ldots, q $ and columns $ r, \ldots, s $.  Also, the notation $ A_{(:,r:s)} $ is defined similarly, except all rows are included in the submatrix.

\section{Continuous time POD with finite dimensional data}
\label{sec:cont_disc_POD}

Next, we consider continuous time POD of a specific class of finite dimensional data that is often generated by Galerkin-type methods for approximating solutions of partial differential equations.   First, in \Cref{subsec:cont_POD_approx} we review our approximate approach from \cite{Fareed17}.  Next, in \Cref{subsec:cont_POD_exact} we show that the SVD of the POD operator is equivalent to the SVD of two different matrices with respect to different weighted inner products.  Below, let $ X $ be a real separable Hilbert space with inner product $ (\cdot,\cdot)_X $, and suppose $ u \in L^{2}( I; X)$, where $ I = (0,T) $ and $ T > 0 $ is fixed.

\subsection{Approximate approach using a Riemann sum}
\label{subsec:cont_POD_approx}

Below, we give a brief summary of our approximate approach in \cite[Section 5]{Fareed17}.  Assume the data $ u $ is known at certain points in time $ 0 = t_1 < t_2 < \cdots < t_{s+1} = T $ and approximate the POD integral operator \eqref{equ:10} using a Riemann sum to get
$$
  K g \approx \sum_{j=1}^s \delta_j \, u_j \, g(t_j),  \quad  \delta_j = t_{j+1} - t_{j},  \quad  u_j = u(t_{j}).
$$
Next, define $ \tilde{u}_j = \delta_j^{1/2} u(t_j) $ and $ h_j = \delta_j^{1/2} g(t_j) $ and approximate the continuous POD operator $ K : L^2(0,T) \to X $ by a discrete POD operator $ \tilde{K} : \mathbb{R}^s \to X $ as follows:
$$
  K g \approx \tilde{K} h := \sum_{j=1}^s \tilde{u}_j \, h_j.
$$

Assume $ u_j \in X $ for each $ j $ is expressed in terms of a finite set of basis functions:
\begin{equation}\label{eqn:data_expanded_basis}
  u_{j} = \sum_{k=1}^{m} U_{k,j} \phi_{k},  \quad  \mbox{for $ j = 1, \ldots, s $,}
\end{equation}
where $ \{ \phi_k \}_{k=1}^m \subset X $ is a linearly independent set.  Define the matrices $ M \in \mathbb{R}^{m \times m} $ and $ U \in \mathbb{R}^{m \times s} $ by their entries $ M_{j,k} := ( \phi_{j} , \phi_{k} )_X $ and $ U_{k,\ell} $, for $ j, k = 1, \ldots, m $ and $ \ell = 1, \ldots, s $.  Also, define the matrix $ \Delta \in \mathbb{R}^{s \times s} $ by $\Delta = \mathrm{diag}(\delta_1, ..., \delta_s)$.  We can use Appendix A.1 in \cite{Fareed17} to see that the core SVD of the discrete POD operator $ \tilde{K} : \mathbb{R}^s \to X $ is obtained from the core SVD of the matrix $ U \Delta^{1/2} : \mathbb{R}^s \to \mathbb{R}^m_M $.  We do not give the details here since similar details will be provided below.

We note that we approximated the POD integral operator and found that the SVD of the resulting discrete POD operator can be obtained by the SVD of the matrix $ U \Delta^{1/2} : \mathbb{R}^s \to \mathbb{R}^m_M $.  This SVD with respect to a weighted inner product can be updated incrementally, as in \cite{Fareed17}.

\subsection{Exact approach assuming the data is piecewise constant in time}
\label{subsec:cont_POD_exact}

Next, we show that the SVD of the POD operator can be obtained without approximation assuming the data $ u $ is piecewise constant in time.  Specifically, we assume the data $ u $ has the form
\begin{equation}\label{eqn:fin_dim_data}
  u(t) = \sum^{s}_{j=1} u_{j}\, \chi_j(t),
\end{equation}
where $ \{u_{j} \} \subset X $ is given as in \eqref{eqn:data_expanded_basis}, $ 0 = t_1 < t_2 < \cdots < t_{s+1} = T $, and the characteristic functions are defined by
\begin{equation}\label{eqn:char_fncs}
  \chi_j(t) = \begin{cases} 1, & t_{j} < t < t_{j+1}, \\ 0, & \text{otherwise.}  \end{cases}
\end{equation}
We first show in \Cref{prop:SVD_two_weights} that the core singular values and singular vectors of $ K $ can be computed by finding the core SVD of a weighted coefficient matrix with respect to two weighted inner products.

As in \Cref{subsec:cont_POD_approx}, throughout this section let the matrices $ M \in \mathbb{R}^{m \times m} $ and $ U \in \mathbb{R}^{m \times s} $ have entries $ M_{j,k} := ( \phi_{j} , \phi_{k})_X $ and $ U_{k,\ell} $, for $ j, k = 1, \ldots, m $ and $ \ell = 1, \ldots, s $.  Also, let $ \Delta \in \mathbb{R}^{s \times s} $ be given by $\Delta = \mathrm{diag}(\delta_1, ..., \delta_s)$, where $ \delta_j = \delta_j = t_{j+1} - t_{j} $ for $ j = 1, \ldots, s $.

\begin{proposition}\label{prop:SVD_two_weights}
Suppose $ \{\phi_{k}\}_{k=1}^{m} \subset X $ are linearly independent, and assume $ u \in L^2(0,T;X) $ is given by \eqref{eqn:data_expanded_basis}-\eqref{eqn:char_fncs}.  Then $ \{ \sigma_i, w_i, v_i \} \subset \mathbb{R} \times \mathbb{R}^s_\Delta \times \mathbb{R}^m_M $ are the core singular values and singular vectors of $ U \Delta : \mathbb{R}^s_\Delta \to \mathbb{R}^m_M $ if and only if $\{ \sigma_{i} , f_{i}, x_{i} \} \subset \mathbb{R} \times L^2( 0, T) \times X $ are the core singular values and singular vectors of $ K : L^{2}( 0, T) \rightarrow X $, where, for all $ i $, $ v_i \in \mathbb{R}^m_M $ and $ x_i \in X $ are related by
\begin{equation}\label{eqn:v_sv}
  x_{i} = \sum_{k=1}^{m} v_{i,k} \, \phi_{k},
\end{equation}
and $ w_i \in \mathbb{R}^s_\Delta $ (with entries $ w_{i,j} $) and $ f_i \in L^2(0,T) $ are related by
\begin{align}
  w_{i,j}  &=  \int^{T}_{0} \delta^{-1}_j \chi_j(t) f_{i}(t) \, dt,\label{eqn:w_sv}\\
  f_i(t)   &=  \sum_{\ell = 1}^s w_{i,\ell} \, \chi_\ell(t).\label{eqn:f_sv}
\end{align}
%
%
\end{proposition}
\begin{proof}
First, since $ \{\phi_{k}\}_{k=1}^{m} \subset X $ is a linearly independent set, we know $ M $ is symmetric positive definite.

Next, assume $ K f_i = \sigma_i x_i $ holds with $ \sigma_i > 0 $, and define $ w_i \in \mathbb{R}^s_\Delta $ as in \eqref{eqn:w_sv}.  Use the definitions of $ K $ in (\ref{equ:10}) and $ u $ in \eqref{eqn:data_expanded_basis}-\eqref{eqn:char_fncs} to get
$$
   \sum_{k=1}^m \bigg( \int^{T}_{0} \sum^{s}_{j=1} U_{k,j} \, \chi_j(t) \, f_i(t) \, dt \bigg) \phi_k = \sigma_i \, x_i.
$$
This implies there exists constants $ v_{i,k} $ so that \eqref{eqn:v_sv} holds.  Substitute \eqref{eqn:v_sv} in the formula above, and then use that $ \{\phi_{k}\}_{k=1}^{m} \subset X $ is a linearly independent set to obtain
$$
  \sum^{s}_{j=1}   U_{k,j} \, \delta_j \, w_{i,j} = \sigma_i \, v_{i,k}.
$$
Therefore,
\begin{equation}\label{eqn:U_sv_eqn1}
  U \Delta w_{i} = \sigma_{i} v_{i}.
\end{equation}
Now assume \eqref{eqn:U_sv_eqn1} holds with $ \sigma_i > 0 $, and define $ f_i \in L^2(0,T) $ by \eqref{eqn:f_sv}.  A similar argument implies that $ K f_i = \sigma_i x_i $, where $ x_i $ is given in \eqref{eqn:v_sv}.

Next, assume $ K^* x_i = \sigma_i f_i $ holds with $ \sigma_i > 0 $, $ x_i $ satisfies \eqref{eqn:v_sv}, and $ w_{i,j} $ is defined by \eqref{eqn:f_sv}.  Use the definitions of $ K^* $ in \eqref{eqn:POD_K_adjoint} and $ u $ in \eqref{eqn:data_expanded_basis}-\eqref{eqn:char_fncs} to get
$$
  \sum_{q=1}^s \sum_{k=1}^m \sum_{\ell=1}^s  v_{i,k} \big( \phi_k, \phi_\ell \big)_X  U_{\ell,q}  \chi_q(t)  =  \sigma_i  f_i(t).
$$
Multiply by $ \chi_j(t) $, integrate over $ (0,T) $, and use $ \int_0^T \chi_q(t) \, \chi_j(t) \, dt = \delta_j \, \delta_{qj} $, where $ \delta_{qj} $ denotes the Kronecker delta symbol, to obtain
$$
  \sum_{k=1}^m \sum_{\ell=1}^s  v_{i,k} M_{k,\ell}  U_{\ell,j}  \delta_j  =  \sigma_i  \delta_j  w_{i,j}.
$$
This gives $ v_i^T M U \Delta = \sigma_i w_i^T \Delta $, or $ U^T M v_i = \sigma_i w_i $.  Since $ (U \Delta)^* = \Delta^{-1} (U \Delta)^T M = \Delta^{-1} \Delta  U^T  M = U^T  M $, we have
\begin{equation}\label{eqn:U_sv_eqn2}
  (U\Delta)^{*} v_{i} = \sigma_{i} w_{i}.
\end{equation}
Now assume \eqref{eqn:U_sv_eqn2} holds with $ \sigma_i > 0 $, and define $ f_i \in L^2(0,T) $ by \eqref{eqn:f_sv}.  Again, a similar argument implies that $ K^* x_i = \sigma_i f_i $, where $ x_i $ is given in \eqref{eqn:v_sv}.

This implies we have
$$
  U\Delta w_{i} = \sigma_{i} v_{i},  \quad  (U\Delta)^{*} v_{i} = \sigma_{i} w_{i}  \quad  \mbox{for all $ i $ with $ \sigma_i > 0 $}
$$
if and only if
$$
  K f_{i} = \sigma_{i} x_{i},  \quad  K^{*} x_{i} = \sigma_{i} f_{i}  \quad  \mbox{for all $ i $ with $ \sigma_i > 0 $.}
$$
Furthermore, $ x_i $, $ v_i $, $ w_i $, and $ f_i $ are related by \eqref{eqn:v_sv}-\eqref{eqn:f_sv}.

Next, assume $\{ \sigma_{i} , f_{i}, x_{i} \} \subset \mathbb{R} \times L^{2}( 0, T) \times X $ are the core singular values and singular vectors of $ K : L^{2}( 0, T) \rightarrow X $.  We show $ \{ w_i \} \subset \mathbb{R}^s_\Delta $ and $ \{ v_i \} \subset \mathbb{R}^m_M $ are both orthonormal sets.  First, using the definition of $ M $ and \eqref{eqn:v_sv} gives
$$
  ( v_i, v_j )_M = ( x_i, x_j )_X = \delta_{ij}.
$$
Next,
\begin{align*}
  ( w_i, w_j )_\Delta &=  \frac{1}{ \sigma_j } \big( w_i, (U\Delta)^* v_j)_\Delta\\
    &=  \frac{1}{ \sigma_j } \big( (U\Delta) w_i, v_j \big)_{M}\\
    &=  \frac{ \sigma_{i} }{ \sigma_j } ( v_i, v_j )_{M}\\
    &=  \frac{ \sigma_{i} }{ \sigma_j } \delta_{ij}  =  \delta_{ij}.
\end{align*}
Therefore $ \{ \sigma_i, w_i, v_i \} \subset \mathbb{R} \times \mathbb{R}^s_\Delta \times \mathbb{R}^m_M $ are the core singular values and singular vectors of $ U\Delta : \mathbb{R}^s_\Delta \to \mathbb{R}^m_M $.

Finally, assume $ \{ \sigma_i, h_i, c_i \} \subset \mathbb{R} \times \mathbb{R}^s_\Delta \times \mathbb{R}^m_M $ are the core singular values and singular vectors of $ U\Delta : \mathbb{R}^s_\Delta \to \mathbb{R}^m_M $.  Similar arguments show $ \{ x_i \} \subset X $ and $ \{ f_i \} \subset L^{2}( 0, T) $ are orthonormal sets, and therefore $\{ \sigma_{i} , f_{i}, x_{i} \} \subset \mathbb{R} \times L^{2}( 0, T) \times X $ are the core singular values and singular vectors of $ K : L^{2}( 0, T) \rightarrow X $.
\end{proof}

In \Cref{subsec:cont_POD_approx}, we rescaled $ \{ u_j \} \subset X $ by the square roots of the time steps to arrive at a different matrix SVD problem.  Again assuming $ u $ is piecewise constant in time as in \eqref{eqn:fin_dim_data}, we have the alternative expansion
\begin{equation}\label{eqn:fin_dim_data2}
  u(t) = \sum^{s}_{j=1} \tilde{u}_{j} \, \tilde{\chi}_j(t),
\end{equation}
where $ \tilde{u}_{j} = \delta_j^{1/2} u_j $ and
\begin{equation}\label{eqn:char_fncs_on}
  \tilde{\chi}_j(t) = \begin{cases} \delta_j^{-1/2}, & t_{j-1} < t < t_{j}, \\ 0, & \text{otherwise.}  \end{cases}
\end{equation}
We show below that the core SVD of $ K $ is equivalent to the core SVD of the matrix $ U \Delta^{1/2} : \mathbb{R}^s \to \mathbb{R}^m_M $.  Note that this is the same matrix SVD we obtained using the Riemann sum approximation approach in \Cref{subsec:cont_POD_approx}.
\begin{proposition}\label{prop:SVD_one_weight}
	Suppose $ \{\phi_{k}\}_{k=1}^{m} \subset X $ are linearly independent, and assume $ u \in L^2(0,T;X) $ is given by \eqref{eqn:data_expanded_basis}-\eqref{eqn:char_fncs}.  Then $ \{ \sigma_i, \tilde{w}_i, v_i \} \subset \mathbb{R} \times \mathbb{R}^s \times \mathbb{R}^m_M $ are the core singular values and singular vectors of $ U \Delta^{1/2} : \mathbb{R}^s \to \mathbb{R}^m_M $ if and only if $\{ \sigma_{i} , f_{i}, x_{i} \} \subset \mathbb{R} \times L^2( 0, T) \times X $ are the core singular values and singular vectors of $ K : L^{2}( 0, T) \rightarrow X $, where, for all $ i $, $ v_i \in \mathbb{R}^m_M $ and $ x_i \in X $ are related by \eqref{eqn:v_sv} and $ \tilde w_i \in \mathbb{R}^s $ (with entries $ \tilde w_{i,j} $) and $ f_i \in L^2(0,T) $ are related by
	\begin{align}
	\tilde w_{i,j}  &=  \int^{T}_{0} \tilde \chi_j(t) f_{i}(t) \, dt,\label{eqn:w_sv_2}\\
	f_i(t)   &=  \sum_{\ell = 1}^s \tilde w_{i,\ell} \, \tilde \chi_\ell(t)  =  \sum_{\ell = 1}^s \delta^{-1/2}_\ell \tilde w_{i,\ell} \, \chi_\ell(t).\label{eqn:f_sv_2}
	\end{align}
	%
	%
\end{proposition}
The proof is similar, and is omitted.  We note that the weighted characteristic functions $ \{ \tilde \chi_j \} \subset L^2(0,T) $ are an orthonormal set, i.e., $ ( \tilde \chi_i, \tilde \chi_j )_{L^2(0,T)} = \delta_{ij} $.  This leads to the removal of the weight on the space $ \mathbb{R}^s $ in the above result.

We also present the connection between the core SVDs of the two matrices $ U \Delta : \mathbb{R}^s_\Delta \to \mathbb{R}^m_M $ and $ U \Delta^{1/2} : \mathbb{R}^s \to \mathbb{R}^m_M $.
\begin{proposition}\label{eqn:SVD_relationships}
  Let $ U \in \mathbb{R}^{m \times s} $, and suppose $ M \in \mathbb{R}^{m \times m} $ and $ \Delta \in \mathbb{R}^{s \times s} $ are symmetric positive definite.  Then the core SVD of $ U \Delta : \mathbb{R}^s_\Delta \to \mathbb{R}^m_M $ is given by $ U \Delta = V \Sigma W^* $ if and only if the core SVD of $ U \Delta^{1/2} : \mathbb{R}^s \to \mathbb{R}^m_M $ is given by $ U \Delta^{1/2} = V \Sigma \tilde{W}^T $, where $ \tilde{W} = \Delta^{1/2} W $.
\end{proposition}
\begin{proof}
  We have $ U \Delta = V \Sigma W^* = V \Sigma W^T \Delta $ where $ V^T M V = I $ and $ W^T \Delta W = I $ if and only if $ U \Delta^{1/2} = V \Sigma W^T \Delta^{1/2} = V \Sigma \tilde{W}^T $, where $ \tilde{W} = \Delta^{1/2} W $, $ V^T M V = I $, and $ \tilde{W}^T \tilde{W} = I $.
\end{proof}
%

\section{Incremental SVD with weighted inner products}
\label{sec:incr_SVD_weights}

In \Cref{sec:cont_disc_POD}, we showed the continuous POD of a certain class of finite dimensional time varying data can be found exactly using the SVD of two different matrices with respect to different weighted inner products.  In this section, we consider incremental approaches to approximating both of these matrix SVDs.  Since incremental SVD algorithms involve approximation, it is possible that one of the two matrix SVDs is computed more accurately than the other.  We show in fact that this is not the case for a specific type of incremental SVD algorithm; specifically, the incremental algorithms for approximating these two matrix SVDs yield equivalent results.

We begin in \Cref{subsec:incr_SVD_two_weights} and follow a similar approach to our earlier work \cite{Fareed17} to develop an incremental SVD with two weighted inner products for $ U \Delta : \mathbb{R}^s_\Delta \to \mathbb{R}^m_M $.  The incremental SVD for $ U \Delta^{1/2} : \mathbb{R}^s \to \mathbb{R}^m_M $ only utilizes one weighted inner product and was developed in \cite{Fareed17}.  In \Cref{sec:equivalent} we show the incremental SVD algorithm for $ U \Delta^{1/2} : \mathbb{R}^s \to \mathbb{R}^m_M $ with one weighted inner product gives an equivalent result as the incremental SVD with two weighted inner products for $ U \Delta : \mathbb{R}^s_\Delta \to \mathbb{R}^m_M $.

Throughout this section we assume $ M \in \mathbb{R}^{m \times m} $ is symmetric positive definite, and $ \Delta \in \mathbb{R}^{s \times s} $ is given by $ \Delta = \mathrm{diag}(\delta_1, \ldots, \delta_s) $, where $ \delta_i > 0 $ for $ i = 1, \ldots, s $.  We assume we know the SVDs of the matrices $ U \Delta : \mathbb{R}^s_\Delta \to \mathbb{R}^m_M $ and $ U \Delta^{1/2} : \mathbb{R}^s \to \mathbb{R}^m_M $, and we focus on updating the SVDs when a new column is added to $ U $ and a new positive diagonal entry $ \delta_{s+1} $ is added to $ \Delta $.  Adding a new column to $ U $ and a new positive diagonal entry to $ \Delta $ corresponds to obtaining the value of the data $ u(t) $ in the new time interval $ t_{s+1} < t < t_{s+2} $.  As discussed in \cite{Fareed17}, the SVDs can be initialized using a single column of data, and then updated incrementally as new data becomes available.

We note that it is possible to further rescale the data in order to remove the weight matrix $ M $ from the inner product for the space $ \mathbb{R}^m_M $.  Since $ M $ is often not diagonal in applications, there are computational disadvantages to performing such a rescaling; see \cite{Fareed17} for a discussion of this issue.  Therefore, we do not consider this type of rescaling here.

\subsection{Incremental SVD with two weighted inner products}
\label{subsec:incr_SVD_two_weights}

Suppose an exact core SVD of $ U\Delta : \mathbb{R}_{\Delta}^{s}\longrightarrow \mathbb{R}_{M}^{m}$ is known, and the goal is to update the core SVD when a new column $ c \in \mathbb{R}^m_M $ is added to $ U $.  First, we prove that the exact core SVD can be updated exactly when no truncation is performed.
\begin{theorem}\label{thm:svd_exact_update}
Suppose $ U \Delta = V \Sigma W^* $ is the exact core SVD of $ U\Delta : \mathbb{R}_{\Delta}^{s} \to \mathbb{R}_{M}^{m}$, where $ V^{T} M V = I $ for $V \in \mathbb{R}^{m \times k}$, $ W^{T} \Delta W = I$ for $W \in \mathbb{R}^{s \times k}$, $ W^{*}  = W^{T} \Delta $, and $\Sigma \in \mathbb{R}^{k \times k}$.  Let $ c \in \mathbb{R}^{m}_M $ and define
\[
  h = c-VV^{*}c,  \quad  p = \| h \|_{M},  \quad
  Q=\begin{bmatrix}
  \Sigma & \delta_{s+1}^{1/2} V^{*}c\\
  0 & \delta_{s+1}^{1/2} p
  \end{bmatrix},
\]
where $ V^* = V^T M $.  If $ p > 0 $ and the standard core SVD of $ Q \in \mathbb{R}^{k+1 \times k+1} $ is given by
  \begin{equation}
  \label{eq:4}
  Q =V _{Q}\,\Sigma_{Q}\,W_{Q}^{T},
\end{equation}
then the core SVD of $ [ \,U  \, c \, ] \Delta_\mathrm{new} : \mathbb{R}_{\Delta_\mathrm{new}}^{s+1} \to \mathbb{R}_{M}^{m} $ is given by
\[
 [\,U \, c \,] \Delta_\mathrm{new} = V_\mathrm{new} \Sigma_{Q} W_\mathrm{new}^{*},\nonumber
\]
where
\[
 V_\mathrm{new} = [\, V \, j \,] ~V_{Q},  \quad  j = h/p,  \quad  W_\mathrm{new} = W_u W_{Q},  \quad  W_u = \begin{bmatrix} W & 0\\ 0 & \delta_{s+1}^{-1/2} \end{bmatrix},
\]
and
\[
W_\mathrm{new}^{*} = W_{u}^{T} \Delta_\mathrm{new},  \quad  \Delta_\mathrm{new} = \mathrm{diag}(\delta_1, \ldots, \delta_{s+1}).
\]
\end{theorem}
\begin{proof}
By the definition of $ j $, we have $ c = V V^* c + jp $.  This gives
\begin{align*}
   [ \,U  \;\; c \, ] \Delta_\mathrm{new}  &=  [\, V \Sigma W^{*} \;\; \delta_{s+1} c \, ]\\
    &=  [ \, V \Sigma W^{*} \;\; \delta_{s+1} (VV^{*}c+jp)  \, ]\\
    &=  [\, V \, j\,] \left[\begin{array}{cc} \Sigma W^T \Delta &  \delta_{s+1} V^{*}c \\ 0 & \delta_{s+1} p  \end{array}\right]\\
  &=  [\, V \,j \,] \left[\begin{array}{cc} \Sigma & \delta_{s+1}^{1/2} V^{*}c \\ 0 & \delta_{s+1}^{1/2} p  \end{array}\right] \left[\begin{array}{cc} W & 0\\ 0 & \delta_{s+1}^{-1/2} \end{array}\right]^T  \left[\begin{array}{cc} \Delta & 0\\ 0 & \delta_{s+1} \end{array}\right]\\
  &=  [\, V \,j \,] \, V _{Q}\,\Sigma_{Q}\,W_{Q}^{T} \left[\begin{array}{cc} W & 0\\ 0 & \delta_{s+1}^{-1/2} \end{array}\right]^T  \left[\begin{array}{cc} \Delta & 0\\ 0 & \delta_{s+1} \end{array}\right]\\
  &=  ([\, V \,j \,] V _{Q})\,\Sigma_{Q}\, \left(\left[\begin{array}{cc} W & 0\\ 0 & \delta_{s+1}^{-1/2} \end{array}\right] W_{Q}\right)^T  \left[\begin{array}{cc} \Delta & 0\\ 0 & \delta_{s+1} \end{array}\right].
\end{align*}
Next, the same argument in the proof of Theorem 4.1 in \cite{Fareed17} shows $ [\, V \,j \,]^T M [\, V \,j \,] = I \in \mathbb{R}^{k+1 \times k+1} $.  Also, since $ W^{*} W = W^{T} \Delta W = I \in \mathbb{R}^{k \times k} $,
\begin{align*}
  \left[\begin{array}{cc} W & 0\\  0 & \delta_{s+1}^{-1/2} \end{array}\right]^{*} \left[\begin{array}{cc} W & 0\\  0 & \delta_{s+1}^{-1/2} \end{array}\right]& = \left[\begin{array}{cc} W & 0\\  0 & \delta_{s+1}^{-1/2} \end{array}\right]^{T} \left[\begin{array}{cc} \Delta & 0\\ 0 & \delta_{s+1} \end{array}\right] \left[\begin{array}{cc} W & 0\\  0 & \delta_{s+1}^{-1/2} \end{array}\right] =  \left[\begin{array}{cc} I & 0\\ 0 & 1 \end{array}\right].
\end{align*}
\Cref{prop:2.3} completes the proof.
\end{proof}

Next, we follow the implementation strategy in \cite[Section 4.2]{Fareed17} to develop the full algorithm.  We only provide a brief summary of various parts of the complete implementation, and refer to \cite[Section 4.2]{Fareed17} for more details and discussion.

\textbf{Initialization.}  We initialize the SVD using a nonzero column of data $ c $ by setting
$$
  \Sigma = \left\Vert \,c\,\right\Vert_{M} = (| c^{T} M c |)^{1/2},  \quad  V = c \Sigma^{-1},  \quad  W = \delta_1^{-1/2},  \quad  \Delta = \delta_1,
$$
where $\delta_1$ is the first time step.\footnote{Although $ M $ is symmetric positive definite, as in \cite[Section 4.2]{Fareed17} absolute values are used since sometimes round off errors may cause $ c^T M c $ to be small and negative.}

\textbf{Exact SVD update.}   Once we have an existing SVD of $ U \Delta : \mathbb{R}^s_\Delta \to \mathbb{R}^m_M $, to update the SVD we first compute $ p = \|c-VV^{*}c \|_{M} $ using the new column $ c $.  In \Cref{thm:svd_exact_update}, $ p = \|c-VV^{*}c \|_{M} $ is assumed to be positive in order to guarantee an exact SVD update.  In practice, we use \Cref{thm:svd_exact_update} for the SVD update only if $ \delta_{s+1}^{1/2} p \geq \mathrm{tol} $, for a given tolerance $ \mathrm{tol} $.

\textbf{Truncation I.}   If instead $ \delta_{s+1}^{1/2} p < \mathrm{tol} $, the final row of $ Q $ is set to zero.  Since $ \delta_{s+1} > 0 $, we set $ p = 0 $ and this implies $ c = V V^* c $.  Using a similar argument to part of the proof of \Cref{thm:svd_exact_update}, setting $ p = j = 0 $, and following a similar approach to \cite[Section 4.2]{Fareed17} provides the SVD update:
$$
V  \longrightarrow  V V_{Q_{(1:k,1:k)}},  \quad  \Sigma  \longrightarrow \Sigma_{Q_{(1:k,1:k)}},  \quad  W  \longrightarrow \begin{bmatrix} W & 0\\ 0 & \delta_{s+1}^{-1/2}\end{bmatrix}  W_{Q_{(:,1:k)}},
$$
where $ V_Q \Sigma_Q W_Q = Q $ is the SVD of $ Q $.  We note in this case the rank of the SVD does not increase.

\textbf{Orthogonalization.} To avoid a loss of orthogonality, we apply a modified $ M $-weighted Gram-Schmidt procedure with reorthogonalization to the columns of $ V $; see \cite[Algorithm 3]{Fareed17}.

\textbf{Truncation II.} To reduce the computational cost and storage, we keep only the singular values and corresponding singular vectors above a user specified tolerance, $ \mathrm{tol}_\mathrm{sv} $.

\textbf{Complete Implementation.}  The incremental SVD update algorithm for two weighted inner products is given in \Cref{algorithm:incrSVD_2weightedinner}.  This algorithm is used every time a new column of data is added.
\begin{algorithm}
\caption{Incremental SVD with two weighted inner products}
\label{algorithm:incrSVD_2weightedinner}
\begin{algorithmic}[1]
  \REQUIRE  $ V \in \mathbb{R}^{m \times k} $, $ \Sigma \in \mathbb{R}^{k \times k} $, $ W \in \mathbb{R}^{\ell \times k}$, $ c \in \mathbb{R}^m $, $ M \in \mathbb{R}^{m \times m} $, $ \Delta \in \mathbb{R}^{\ell \times \ell} $, $\delta > 0 $, $ \mathrm{tol} $, $ \mathrm{tol}_\mathrm{sv} $,
  \\
  \smallskip
  \% Prepare for SVD update
  \smallskip
  \STATE  $d=V^{T} M c$,\, $ p = \mathrm{sqrt}( |( c-Vd )^{T} M ( c-Vd )| )$,\, $ \alpha = \delta^{1/2} $ 
    \IF{$(p \, \alpha < \mathrm{tol})$} 
  \STATE  $ Q = \begin{bmatrix}
        \Sigma & d \,\alpha \end{bmatrix} $
    \ELSE
        \STATE  $ Q = \begin{bmatrix}
        \Sigma & d \,\alpha\\
        0 & p \,\alpha
   \end{bmatrix} $
    \ENDIF
  \STATE  $[\,V_{Q},\Sigma_{Q},W_{Q}\,]=\mathrm{svd}(Q)$ 
\STATE  $ \Delta_\mathrm{new} = \begin{bmatrix}
        \Delta & 0\\
        0 & \delta
   \end{bmatrix} $
  \\
  \smallskip
  \% SVD update
  \smallskip
  \IF{$(p \, \alpha < \mathrm{tol})$ or $(k \geq  m)$}
  \STATE $ V = V V_{Q_{(1:k,1:k)}}$, $\Sigma = \Sigma_{Q_{(1:k,1:k)}}$, $ W =\begin{bmatrix} W & 0\\ 0 & \delta^{-1/2} \end{bmatrix} W_{Q_{(:,1:k)}} $
  \ELSE
  \STATE $ j=(c-Vd) / p $
  \STATE $ V= [V \, j] V_{Q} $, $ \Sigma = \Sigma_{Q}$,  $ W = \begin{bmatrix}
    W & 0\\
    0 & \delta^{-1/2}
    \end{bmatrix} W_{Q} $
  \STATE $ k=k+1 $
  \ENDIF
  \\
  \smallskip
  \% Orthogonalize if necessary (see Algorithm 3 in \cite{Fareed17})
  \smallskip
  \IF{( $ | V_{(:,\mathrm{end})}^T M V_{(:,1)} | > \mathrm{min}(\mathrm{tol} , \mathrm{tol}\times m) $)}
  \STATE $ V = \mathrm{modifiedGSweighted}(V,M) $
  \ENDIF
  \\
  \smallskip
  \% Neglect small singular values: truncation
  \smallskip
  \IF{$(\Sigma_{(r,r)}  > \mathrm{tol}_\mathrm{sv})$ and $(\Sigma_{(r+1,r+1)} \leq \mathrm{tol}_\mathrm{sv})$}     
  \STATE $ \Sigma = \Sigma_{(1:r,1:r)}$, \quad  $V = V_{(:,1:r)}$, \quad  $W = W_{(:,1:r)}$  
  \ENDIF
  \RETURN  $ V $, $ \Sigma $, $ W $, $ \Delta_\mathrm{new} $
\end{algorithmic}
\end{algorithm}
%

\subsection{Incremental SVD with one weighted inner product is equivalent}
\label{sec:equivalent}

Next, we consider the incremental SVD algorithm for $ U \Delta^{1/2} : \mathbb{R}^s \to \mathbb{R}^m_M $ with one weighted inner product from \cite{Fareed17}.  However, instead of recalling that algorithm here, for brevity we simply modify \Cref{algorithm:incrSVD_2weightedinner} to work for the present case.

Since the space $ \mathbb{R}^s $ does not utilize a weighted inner product, we start with an SVD $ U \Delta^{1/2} = V \Sigma \tilde{W}^T $, where $ V^T M V = I $ and $ \tilde{W}^T \tilde{W}= I $.  Recall from \Cref{eqn:SVD_relationships} that $ \tilde{W} = \Delta^{1/2} W $, where $ U \Delta = V \Sigma W^* $ is the SVD of $ U \Delta : \mathbb{R}^s_\Delta \to \mathbb{R}^m_M $.

Let $ c_j $ denote the $ j $th column of $ U $, and let $ \Delta = \mathrm{diag}(\delta_1, \ldots, \delta_s) $, as before.  Then the matrix $ U \Delta^{1/2} $ has $ j $th column $ \tilde{c}_j = \delta_j^{1/2} c_j $.  To update the SVD of $ U \Delta^{1/2} : \mathbb{R}^s \to \mathbb{R}^m_M $, we simply use \Cref{algorithm:incrSVD_2weightedinner} with $ (\Delta,\delta,c,W) $ replaced by $ (I,1,\tilde{c},\tilde{W}) $ in the algorithm.  Since the $ j $th column of $ U $ is $ \tilde{c}_j = \delta_j^{1/2} c_j $ and the matrix $ Q $ depends only on $ V $ and $ \tilde{c} $, it is easily checked that the algorithm produces the same $ V $ and the same $ \Sigma $ as produced by \Cref{algorithm:incrSVD_2weightedinner} in the two weighted inner product case.  Furthermore, it is also easily checked that the $ \tilde{W} $ produced by the algorithm still satisfies $ \tilde{W} = \Delta^{1/2} W $, where $ W $ is the update produced by \Cref{algorithm:incrSVD_2weightedinner} in the two weighted inner product case.

Therefore, the two incremental SVD approaches yield equivalent results.

%
\section{Conclusion}
\label{sec:conclusion}

We revisited our earlier work \cite{Fareed17} on an incremental POD algorithm for PDE simulation data.  For the case of time varying data, in \cite{Fareed17} we developed an approximate Riemann sum approach to arrive at a discrete time POD problem and an incremental POD algorithm.  In this work, we considered an alternative viewpoint to develop and analyze incremental POD approaches for time varying data.   We considered piecewise constant in time data taking values in a Hilbert space, where each data snapshot is expanded in a fixed basis.  We first showed that the POD of this data is exactly determined by the SVD of two different matrices with respect to different weighted inner products.  The two different SVDs come from two different ways of expressing the piecewise constant in time data.  Next, we developed incremental SVD algorithms for each case.  Since the incremental algorithms compute approximate SVDs, it was possible that one approach could be more accurate than the other; however, we showed that the incremental SVD algorithms produce equivalent results.  Therefore, the two different ways of expressing the data did not lead to different incrementally computed POD modes for the data.

One benefit of this result is that the error analysis of the discrete time incremental POD algorithm (with one weighted inner product) in \cite{FareedSingler18pp} is directly applicable to continuous time case, assuming the data is piecewise constant in time and is expressed using the weighted characteristic functions as in \eqref{eqn:fin_dim_data2}-\eqref{eqn:char_fncs_on}.  

Furthermore, we note that it may be possible to extend the incremental SVD approach developed here for the case of two weighted inner products to treat time varying data that is not piecewise constant in time.  We leave this to be considered in the future.


\bibliographystyle{plain}
\bibliography{incremental_POD_continuous}

\end{document}